\documentclass[a4paper,11pt]{amsart}

\usepackage{amssymb,amscd,amsthm,mathrsfs}
\usepackage[ansinew]{inputenc}
\usepackage[centertags]{amsmath}
\usepackage{graphicx,psfrag}

\newtheorem*{teo*}{Theorem}

\newtheorem{teo}{Theorem}[section]

\newtheorem{cor}{Corollary}[section]
\newtheorem{lema}{Lemma}[section]
\newtheorem{prop}{Proposition}[section]

\newcommand{\bi}{\begin{itemize}}
\newcommand{\ei}{\end{itemize}}

\theoremstyle{definition}

\theoremstyle{remark}
\newtheorem{obs}[]{Remark}

\newcommand{\T}{\mathbb{T}}
\newcommand{\R}{\mathbb{R}}
\newcommand{\Z}{\mathbb{Z}}
\newcommand{\N}{\mathbb{N}}
\newcommand{\Q}{\mathbb{Q}}

\newcommand{\homeo}{\textrm{Homeo}_0(\T^2)}

\author{Alejandro Passeggi}
\address{TU Dresden, Institute for Analysis, Emmy Noether Research Group 'Low dimensional and non autonoumos Dynamics'.}
\email{alepasseggi@gmail.com}

\title{Rational Polygons as rotation sets of generic homeomorphisms of the two torus}

\begin{document}

\begin{abstract}

We prove the existence of an open and dense set $\mathcal{D}\subset\textrm{Homeo}_0(\T^2)$ (where $\textrm{Homeo}_0(\T^2)$ denotes the set of toral homeomorphisms homotopic to the identity) such that the rotation set of any element in $\mathcal{D}$ is a rational polygon. We also extend this result to the set of axiom A diffeomorphisms in $\homeo$. Further we observe the existence of minimal sets whose rotation set is a non-trivial segment, for an open set in $\homeo$.

\end{abstract}

\maketitle

\begin{section}{Introduction.}\label{intro}

Rotation theory in the two dimensional torus can be seen as a first attempt to generalize Poincar\'{e}'s combinatorial theory on the dynamics of orientation-preserving circle homeomorphisms. In this theory, a complete description of the qualitative dynamics associated to $f\in\textrm{Homeo}_+(\mathbb{S}^1)$ (where $\textrm{Homeo}_+(\mathbb{S}^1)$ denotes the set of orientation-presearving homeomorphisms) can be given via the so called rotation number $\rho(F)$, which is defined by $\rho(F):=\lim_n \frac{F^n(x)-x}{n}$ for some $F$ lift of $f$ and $x\in\R$, but does not depend on $x$.

In the context of the two dimensional torus one considers the space of homeomorphisms homotopic to the identity $\homeo$. However since it is impossible to ensure the independence of the limit on $x$ or even the existence as in the circle case, one has to consider the set of all possible limits. Let $\T^2:=\R^2/\Z^2$ be the two torus, $\pi:\R^2\rightarrow \R^2/_{\Z^2}$ be the projection, $f\in\homeo$, $C\subset \T^2$ a compact and $f$-invariant set and $F:\R^2\rightarrow \R^2$ lift of $f$. The rotation set on $C$ is defined by

$$ \rho_C(F)=\left\{v\in\R^2: v=\lim_i\frac{F^{n_i}(x_i)-x_i}{n_i}:x_i\in\pi^{-1}(C), n_i\nearrow +\infty\right\} $$

We refer to $\rho_{\T^2}(F)$ as the rotation set of $f$ and denote this by $\rho(F)$. Further we define $\rho(f):=\rho(F)(\mbox{ mod }\Z^2)$. The following fundamental result is due to M. Misiurewicz and K. Zieman (\cite{MiZi1}).

\begin{teo}\label{convexity}

$\rho(F)$ is a compact and convex set for any lift $F$.

\end{teo}

We also have the following properties for the rotation set:

\begin{itemize}

\item[(1)] $\rho_C(F)$ is a compact set.

\item[(2)] If there exist a map $h:\T^2\rightarrow\T^2$ homotopic to the identity such that $h\circ g=f \circ h$ for some $f,g\in\homeo$ we say that $g$ is semi-conjugate to $f$ by $h$. In this case, we have $\rho_C(f)=\rho_{h(C)}(g)$ for any closed and $g$-invariant set $C$.

\item[(3)] Given a compact set $A\subset\R^2$, we denote by $\textrm{conv}(A)$ to the convex hull of $A$. Then we have that $\textrm{conv}({\rho_C(F)})=\textrm{conv}(\rho_{\Omega(f)\cap C}(F))$ where $\Omega(f)$ is the non-wandering set of $f$ (see section \ref{theset}).

\item[(4)] For every $n\in\Z$, $\rho_C(F^n)=n\cdot\rho_C(F)$.

\end{itemize}

Properties 1,2, and 4 are not difficult to prove. For property 3 some ergodic-theoretical concepts related to the rotation set need to be introduced, see \cite{MiZi1} and \cite{Da}.

After Misiurewicz's and Zieman's result, the theory has been developed by studying the three following cases:

\begin{itemize}

\item[(1)] the rotation set is given by a single point;

\item[(2)] the rotation set is given by a segment;

\item[(3)] the rotation set has non-empty interior.

\end{itemize}

Interesting dynamical implications have been found for each one of the cases  (see for instance \cite{F1,F2,MiZi2,LiMa,Da,Ja}), and still are studied. A rather different problem is to determine which convex sets of the plane can be realized as the rotation set of some $f\in\homeo$.

For this problem, there are still no answers to fundamental questions. If we see the problem in each of the three cases we have:

\begin{itemize}

\item For the first case, any single point in the plane can obviously be realized as a rotation set of some homeomorphism.

\item For the second case there is a conjecture due J. Franks and M. Misiurewicz, which states that the segments which can be realized as rotation set either contain a rational point as an endpoint or they contain infinitely many rational points. This conjecture is still open.

\item For the third case it is not known whether all convex set in the plane with non empty interior can be the rotation set of some homeomorphism. In this direction J. Kwapisz has shown that every polygon in the plane with rational vertices can be realized as a rotation set (see \cite{Kw1}). Also in \cite{Kw2} Kwapisz constructs a diffeomorphism in $\homeo$ such that its rotation set contains countably many extremal points.

\end{itemize}

In this work we prove first that generically the rotation set is given by a rational polygon. Here, a rational polygon is any convex set in the plane with finitely many extremal points which are rational. Note that this could be even segments and single points, to which we refer as degenerated rational polygons. We consider the $C^0$ topology in $\homeo$, and we obtain the following result.

\begin{teo*}

There exists an open and dense set $\mathcal{D}\subset\homeo$ such that for every $f\in\mathcal{D}$ the rotation set $\rho(F)$ is a rational polygon (possibly degenerate).

\end{teo*}

This result (which is restated below as Theorem \ref{TEO1}) can be seen as a generalization of the corresponding result in the circle theory, which asserts that for an open and dense set in $\textrm{Homeo}_0(S^1)$ the rotation set of any element belonging to this set has a rational rotation number. Further this result generalizes S. Zanata's result \cite{Zana}, in which is proved that if for some homoemorphism $f$ the rotation set has an irrational extremal point, then the rotation set can be modified with arbitrary small perturbations of $f$.

\ The idea is to see that the rotation set of certain \textit{axiom A diffeomorphism} are rational polygons (for which we were inspired by the articles \cite{Kw1} and \cite{Zi}), and then making use of the \textit{genericity} and \textit{robustness} of these diffeomorphism to construct the set $\mathcal{D}$. After having obtained this first result, we extend it to all Axiom A diffeomorphism in $\homeo$ by using techniques from hyperbolic dynamics (see Section \ref{Axasec}).

\begin{teo*}

Let $f\in \homeo$ be an axiom A diffeomorphism. Then $\rho(F)$ is a rational polygon.

\end{teo*}

Finally, we consider rotation sets on minimal sets of elements in $\homeo$. As the definition of minimal set can be seen as a natural generalization of periodic orbits, one could guess that the rotation set on a minimal set is given by a unique point, as in the case of periodic orbits. We see here that this is not true. In fact, we observe in general that the rotation sets on minimal sets are continua, and further we prove the following result (see Corollary \ref{abun1}):

\begin{teo*}

There exist an open set $\mathcal{D}'\subset \homeo$, such that any $f\in\mathcal{D}'$ has a minimal set $\mathcal{M}$ with $\rho_{\mathcal{M}}(F)$ given by a non-trivial segment. Furthermore, $\mathcal{D}'$ is dense in the set of homeomorphisms whose rotation set has non-empty interior.

\end{teo*}

\textbf{Acknowledgments:} I would like to thank Tobias J\"ager, Martin Sambarino and Rafael Potrie for very useful comments and discussions, which have inspired important ideas for this article.

\end{section}

\begin{section}{Symbolic representation and rotation set.}\label{symbolicrep}

Let us consider a closed and invariant set $C\subset \T^2$ of $f\in\homeo$, and a family $\mathcal{S}=\{S_0,...,S_M\}$ of pairwise disjoint sets contained in $\T^2$ such that:

\begin{itemize}

\item $S_i$ is homeomorphic to the unit square  for every $i=0,...,M$;

\item $C\subset \bigcup_{i=0}^M S_i$.

\end{itemize}

We say that $\mathcal{S}$ is a finite partition of $C$ by rectangles. For the sake of simplicity we refer to $\mathcal{S}$ as a partition. As usual, for every partition $\mathcal{S}$ we can associate the space of sequences of symbols $\Omega_{\mathcal{S}}\subset\{0,...,N\}^{\Z}$ defined as

\begin{equation*}
\Omega_{\mathcal{S}}=\left\{ \xi\in \{0,...,N\}^{\Z}:\exists\ x\in C \mbox{ such that } f^n(x)\in S_{\xi(n)}\forall n\in\Z\right\}.
\end{equation*}

The space $\Omega_{\mathcal{S}}$ is closed and invariant under the shift $\sigma:\{0,...,N\}^{\Z}\rightarrow\{0,...,N\}^{\Z}$, defined by $\sigma(\xi)(n):=\xi(n+1)$ for every $n\in\Z$. Furthermore, there is a natural onto map $h:C\rightarrow\Omega_{\mathcal{S}}$ given by the relation $f^n(x)\in S_{h(x)(n)}$ for every $n\in\Z$, which semi-conjugates $f\mid_C$ with $\sigma\mid_{\Omega_{\mathcal{S}}}$ (that is $h\circ f=\sigma\circ h$ holds in $C$).

\begin{subsection}{Symbolic representation and displacement.}\label{goodpar}

We say that an open, connected and bounded set $D\subset\R^2$ is a trivial domain if $\pi_{|D}:D\rightarrow \pi(D)$ is a homeomorphism. If $C$ and $\mathcal{S}=\{S_1,...,S_M\}$ are as above, we say that $\mathcal{S}$ is a rotational partition if there exists a trivial domain $D$ such that:

\begin{itemize}

\item[(1)] $\pi(D)\supset \bigcup_{i=0}^M S_i$;

\item[(2)] if we consider $\tilde{S}_i:=\pi^{-1}(S_i)\cap D$ for $i=0,...,M$ and a lift $F$ of $f$, then for each $i=1,...,M$ there is a unique vector $s_i\in\Z^2$ such that $F(\tilde{S_i})\subset D+s_i$.

\end{itemize}

Notice that a roational partition has the following properties: (1) $\tilde{S}_i:=\pi^{-1}(S_i)$ is a rectangle, and (2) the image of any rectangle $\tilde{S}_i$ by the lift $F$ meets at most one integer translation of any rectangle $\tilde{S}_j$.

Given a surface $S$, $A\subset S$ and $x\in A$, we denote by $[A]_x$ the connected component of $x$ in $A$. Let us consider a rotational partition $\mathcal{S}=\{S_0,...,S_N\}$ of a closed and $f$-invariant set $C$, and a fixed $n\in\N$. Then for every $i\in\{0,...,N\}$ we define the following families of subsets of $\T^2$: 

$$\mathcal{I}_i=\left\{\pi([K]_x):K=F^n(\tilde{S}_i)\cap \pi^{-1}\left(\bigcup_{j=0}^N S_j\right),x\in\pi^{-1}(C)\cap K\right\}\mbox{ and } f^n_*(\mathcal{S})=\bigcup_{i=0}^N\mathcal{I}_i.$$

\ For a partition $\mathcal{S}$ we define the positive number $d_{\mathcal{S}}:=\max\{\textrm{diam}(S_i):S_i\in\mathcal{S}\}$. The following are simple observations. We omit the proofs.

\begin{obs}\label{smallsize}

For every closed set $C$ which admits a partition $\mathcal{R}$, there exists $\varepsilon_0>0$ such that every partition $\mathcal{S}$ of $C$ with $d_{\mathcal{S}}<\varepsilon_0$ is a rotational partition.

\end{obs}

\begin{obs}\label{pushisrot}

If $f^n_*(\mathcal{S})$ is a partition of $C$, then it is a rotational partition for $f^n$.

\end{obs}

Given a partition $\mathcal{S}$ we consider $\Omega_{\mathcal{S}}$ and the semi-conjugacy $h$.

\begin{prop}\label{disp}

Let $\mathcal{S}$ be a rotational partition of $C$. Then, for every $x\in\pi^{-1}(C)$ and $n\in\N$ we have $F^n(x)\in B(x+\sum_{j=0}^{n-1}s_{i_j},2\cdot d_{\mathcal{S}})$, where $i_j=h(\pi(x))(j)$ for $j=0,...,n-1$.

\end{prop}

\begin{proof}

Let $R=2\cdot d_{\mathcal{S}}$. For a fixed $n>0$ and $x\in\pi^{-1}(C)$ we have that $f^j(\pi(x))\in S_{i_j}$ where $i_j=h(\pi(x))(j)$ for every $j=0,...,n$. Therefore for every $j=0,...,M$ $F^j(x)\in \tilde{S}_{i_j}+v_j$, for some $v_j\in\Z^2$ ($\tilde{S}_{i}$ as in the definition of the rotational partition). Moreover, for every $j=1,...,M$ we have $v_j=v_{j-1}+s_{i_{j-1}}$ and $v_0=u$ with $u\in\Z^2$ such that $x\in\tilde{S}_{i_0}+u$. This implies by replacement that $F^{n}(x)\in \tilde{S}_{i_n}+s_{i_{n-1}}+...+s_{i_0}+u$, hence $F^{n}(x)\in B(x+\sum_{j=0}^{n-1}s_{i_j},R)$ where $i_j=h(\pi(x))(j)$ for $j=0...n-1$.

\end{proof}

Now we introduce some previous notation before stating an important corollary about the rotation set of $F$ on $C$. We define the set of finite words $\Omega^f_{\mathcal{S}}\subset \Omega_{\mathcal{S}}$ as

\begin{equation*}
\Omega^f_{\mathcal{S}}=\left\{\alpha:\{0,1,...,l_{\alpha}-1\}\rightarrow \{0,...,M\}:l_{\alpha}\in\N\mbox{ and for some }\xi\in\Omega_{\mathcal{S}}\mbox{ }\alpha(i)=\xi(i)\forall i=0,....,l_{\alpha}-1\right\}
\end{equation*}

For $\alpha\in\Omega^f_{\mathcal{S}}$ we call $l_{\alpha}$ the length of $\alpha$. Further, we define a function $\psi:\Omega^f_{\mathcal{S}}\rightarrow \R^2$ by $\psi(\alpha)=\sum_{j=0}^{l_{\alpha}}s_{\alpha(j)}$.

\begin{cor}\label{rs} Let $\mathcal{S}$ be a rotational partition of $C$, then

$\rho_{C}(F)=\left\{v\in\R^2|\ \exists \ (\alpha_i)_{i\in\N}\subset \Omega^f_{\mathcal{S}}\mbox{ with }l{\alpha_i}\nearrow +\infty\mbox{ and }v=\lim_{i}\frac{\psi(\alpha_i)}{l_{\alpha_i}}\right\}$

\end{cor}

\end{subsection}

\begin{subsection}{Finite Markov partitions and the rotation set.}\label{Simbolicstuff}

We call a function
\\$\tau:\{0,...,M\}\rightarrow \mathcal{P}(\{0,...,M\})$ with $\emptyset\notin Im(\tau)$ transition law. Given a transition law we define $\Omega_{\tau}\subset \{0,...,M\}^{\Z}$ as the set of sequences $\xi$ such that $\xi(n+1)\in \tau(\xi(n))$ for every $n\in\Z$. Let us consider a partition $\mathcal{S}=\{S_0,...,S_M\}$. We say that $\mathcal{S}$ is a (finite) Markov partition if $\Omega_{\mathcal{S}}=\Omega_{\tau}$ for some transition law $\tau$. In this subsection we want to prove that the convex hull of $\rho_{C}(F)$ is a rational polygon whenever $C$ admits a rotational Markov partition.

\ Now if we consider the space of finite words $\Omega^f_{\mathcal{S}}$ for a Markov partition and two elements $\alpha_1,\alpha_2\in\Omega^f_{\mathcal{S}}$ with lengths $l_1,l_2$ respectively such that $\tau(\alpha_1(l_1))=\alpha_2(0)$, the function $\alpha_1\alpha_2:\{0....l_1+l_2\}\rightarrow\{0,...,M\}$ given by $\alpha_1\alpha_2(i)=\alpha_1(i)$ if $i=0...l_1-1$, and $\alpha_1\alpha_2(i)=\alpha_2(i-(l_1))$ if $i=l_1,...,l_1+l_2-1$ is an element of $\Omega^f_{\mathcal{S}}$. We call $\alpha_1\alpha_2(i)$ the concatenation of $\alpha_1$ and $\alpha_2$.

In particular, if $\nu\in\Omega^f_{\mathcal{S}}$ verifies that $\nu(0)\in \tau(\nu(l_{\nu}-1))$ then one can consider for each $n\in\N$ the word $\nu^n\in\Omega^f_{\mathcal{S}}$ given by $\nu^n=\underbrace{\nu...\nu}_{n-times}$. We call this kind of elements periodic words and denote by $\Omega_{\mathcal{S}}^{Per}\subset \Omega^f_{\mathcal{S}}$ the set of periodic words. Then is easy to see that for any $\nu\in\Omega_{\mathcal{S}}^{Per}$ the vector $\frac{\psi(\nu)}{l_{\nu}}$ belongs to $\rho_C(F)$. We also add an extra word $\theta$ to $\Omega^f_{\mathcal{S}}$ that we call the trivial word which verifies $\theta\alpha=\alpha\theta=\alpha$ for every $\alpha\in \Omega^f_{\mathcal{S}}$.

\ In what follows we denote by $\textrm{cl}[A]$ the closure of $A\subset\R^2$, by $\textrm{conv}(A)$ the convex hull of $A$ and let $\rho^{Per}_C(F)=\left\{v\in\R^2|\ \exists \ (\alpha_i)_{i\in\N}\subset \Omega^{Per}_{\mathcal{S}}\mbox{ with }v=\lim_{i}\frac{\psi(\alpha_i)}{l_{\alpha_i}}\right\}$.

\begin{lema}\label{prev}

Suppose that $\mathcal{S}$ is a rotational Markov partition of $C$. Then $\rho_C(F)\subset \textrm{cl}[\textrm{conv}(\rho^{Per}_C(F))]$.

\end{lema}

\begin{proof}

Let $v\in\rho_C(F)$. Then by Corollary \ref{rs} we have $v=\lim_i\frac{\psi(\alpha_i)}{l_{\alpha_i}}$ with $\alpha_i\in \Omega^f_{\mathcal{S}}$ for every $i\in\N$ and $l{\alpha_i}$ strictly increasing to $+\infty$. For a fixed $\varepsilon>0$, we want to show the existence of an element of $\textrm{conv}(\rho^{Per}_C(F))$ which is $\varepsilon$-close to $v$. We consider $k_{\varepsilon}\in\N$ which verifies that for every $i>k_{\varepsilon}$ we have $l_{\alpha_i}>M$ and $\bigl\lVert \frac{\psi(\alpha_i)}{l_{\alpha_i}}- \frac{\psi(\alpha_i)-\psi(\beta)}{l_{\alpha_i}-l_{\beta}} \bigl\lVert<\frac{\varepsilon}{2}$ for every $\beta\in\Omega^f_{\mathcal{S}}$ with $l_{\beta}\leq M$.

\ Let $i>k_{\varepsilon}$ such that $\bigl\lVert v-\frac{\psi(\alpha_i)}{l_{\alpha_i}}\bigl\lVert< \frac{\varepsilon}{2}$. Since $l_{\alpha_i}>M$ there exist at least two different elements $j_1,j_2\in\{0,...,l_{\alpha_i}\}$ such that $\alpha_{i}(j_1)=\alpha_{i}(j_2)$. This implies that $\alpha_i=\gamma_1\nu_0 \gamma_2$ for some $\gamma_1,\gamma_2\in \Omega^f_{\mathcal{S}}$ with at least one different of $\theta$ and $\nu_0\in\Omega_{\mathcal{S}}^{Per}$. Moreover since $\nu_0$ is a periodic word $\alpha_0=\gamma_1\gamma_2\in \Omega^f_{\mathcal{S}}$ is a well defined non trivial word and $\psi(\alpha_i)=\psi(\alpha_0)+\psi(\nu_0)$ with $l_{\nu_0}+l_{\alpha_{0}}=l_{\alpha_{i}}$. Now if $l_{\alpha_{0}}>M$ we can do the same as we did with $\alpha_i$ and produce $\nu_1\in\Omega_{\mathcal{S}}^{Per}$, $\alpha_1\in\Omega^f_{\mathcal{S}}$ such that $\psi(\alpha_i)=\psi(\alpha_1)+\psi(\nu_1)+\psi(\nu_0)$ with $l_{\nu_0}+l_{\nu_1}+l_{\alpha_{1}}=l_{\alpha_{i}}$.

\ If we apply this argument recursively we obtain $\nu_0,...,\nu_r\in\Omega_{\mathcal{S}}^{Per}$ and $\alpha_f\in\Omega^f_{\mathcal{S}}$ with $l_{\alpha_f}\leq M$ such that $\psi(\alpha_i)=\sum_{k=0}^r\psi(\nu_k)+\psi(\alpha_f)$ and $\sum_{k=0}^r l_{\nu_k}+l{\alpha_f}=l_{\alpha_i}$. Hence by definition of $k_{\varepsilon}$ follows that $\bigl\lVert \frac{\psi(\alpha_i)}{l_{\alpha_i}} - \frac{\sum_{k=0}^r\psi(\nu_k)}{\sum_{k=0}^r l_{\nu_k}}  \bigl\lVert<\frac{\varepsilon}{2}$ which implies that $\bigl\lVert v - \frac{\sum_{k=0}^r\psi(\nu_k)}{\sum_{k=0}^r l_{\nu_k}}  \bigl\lVert<\varepsilon$. Since $w=\frac{\sum_{k=0}^r\psi(\nu_k)}{\sum_{k=0}^r l_{\nu_k}}$ belongs to $\textrm{conv}(\rho^{Per}_C(F))$, this finishes the proof.

\end{proof}

As we have said the goal of this subsection is prove that $\textrm{conv}({\rho_C(F)})$ is a rational polygon whenever $C$ admits a rotational Markov partition $\mathcal{S}$. Due to Lemma \ref{prev} it is sufficient to prove that $\textrm{conv}(\rho_C^{Per}(F))$ is a rational polygon. The next observation is rather simple, so we do not give the proof.

\begin{obs}\label{prevobs}

Suppose that there exists a constant $K>0$ such that for every $w\in\R^2$ the function
\\ $<w, . >:\rho_C^{Per}(F)\rightarrow \R$ is maximized by $\frac{\psi(\nu)}{l_{\nu}}$ with $\nu\in\Omega_{\mathcal{S}}^{Per},l_{\nu}\leq K$. Then, $\textrm{conv}(\rho_C^{Per}(F))$ is a rational polygon with vertices in the finite set $\left\{\frac{\psi(\nu)}{l_{\nu}}:\nu\in\Omega_{\mathcal{S}}^{Per},l_{\nu}\leq K\right\}$.

\end{obs}

\begin{teo}

Suppose that $\mathcal{S}=\{S_0,...,S_M\}$ is a rotational Markov partition for $C$. Then $\textrm{conv}(\rho_C^{Per}(F))$ is a rational polygon with vertices in $\left\{\frac{\psi(\nu)}{l_{\nu}}:\nu\in\Omega_{\mathcal{S}}^{Per},l_{\nu}\leq M\right\}$.

\end{teo}

\begin{proof}

By definition $\rho_C^{Per}(F)$ is contained in the compact set $\rho_C(F)$. Hence, for a fixed $w\in\R^2$ there exists $L=\sup\left\{\langle w,v\right\rangle:v\in \rho_C^{Per}(F)\}=\sup\left\{\frac{\left\langle w,\psi(\nu)\right\rangle}{l_{\nu}}:\nu\in\Omega_{\mathcal{S}}^{Per}\right\}$. Our goal is to prove that $L=\sup\left\{\frac{\left\langle w,\psi(\nu)\right\rangle}{l_{\nu}}:\nu\in\Omega_{\mathcal{S}}^{Per},l_{\nu}\leq M\right\}$ where $M$ is the size of the partition. Due to the last remark, this implies the desired result. Hence for every $\varepsilon>0$ we want to find an element $\nu_{\varepsilon}\in\Omega_{\mathcal{S}}^{Per}$ with $l_{\nu_{\varepsilon}}\leq M$ such that $L-\frac{\left\langle w,\psi(\nu_{\varepsilon})\right\rangle}{l_{\nu_{\varepsilon}}}<\varepsilon$.

\ For a fixed $\varepsilon>0$ we consider $\nu_0\in\Omega_{\mathcal{S}}^{Per}$ such that $L-\left\langle w,\frac{\psi(\nu_0)}{l_{\nu_0}}\right\rangle <\varepsilon$. If $l_{\nu_0}\leq M$ we are done. If $l_{\nu_0}> M$ there exists $\nu_0'\in\Omega_{\mathcal{S}}^{Per}$ and $\alpha_1,\alpha_2\in\Omega_{\mathcal{S}}^{f}$ with at least one of the two different from $\theta$ such that $\nu_0=\alpha_1\nu_0'\alpha_2$, $\alpha_1\alpha_2\in\Omega_{\mathcal{S}}^{Per}$ and $l_{\alpha_1}+l_{\alpha_2}+l_{\nu_0'}=l_{\nu_0}$. Therefore $\frac{\left\langle w,\psi(\nu_0)\right\rangle}{l_{\nu_0}}=\frac{\left\langle w,\psi(\alpha_1\alpha_2)\right\rangle+\left\langle w,\psi(\nu_0')\right\rangle}{l_{\alpha_1}+l_{\alpha_2}+l_{\nu_0'}}$ which implies that either $\frac{\left\langle w,\psi(\nu_0)\right\rangle}{l_{\nu_0}}\leq \frac{\left\langle w,\psi(\nu_0')\right\rangle}{l_{\nu_0'}}$ or $\frac{\left\langle w,\psi(\nu_0)\right\rangle}{l_{\nu_0}}\leq \frac{\left\langle w,\psi(\alpha_1\alpha_2)\right\rangle}{l_{\alpha_1\alpha_2}}$. Thus there exist $\nu_1\in\Omega_{\mathcal{S}}^{Per}$ such that $\frac{\left\langle w,\psi(\nu_0)\right\rangle}{l_{\nu_0}}\leq \frac{\left\langle w,\psi(\nu_1)\right\rangle}{l_{\nu_1}}$ and $l_{\nu_1}< l_{\nu_0}$.

\ Applying the above argument recursively, we eventually find a periodic word $\nu_{\varepsilon}\in \Omega_{\mathcal{S}}^{Per}$ such that $l_{\nu_{\varepsilon}}\leq M$ and $L-\frac{\left\langle w,\psi(\nu_{\varepsilon})\right\rangle}{l_{\nu_{\varepsilon}}}<\varepsilon$.

\end{proof}

\begin{cor}\label{polygon}

Suppose that $\mathcal{S}$ is a rotational Markov partition of $C$. Then $\textrm{conv}(\rho_C(F))$ is a rational polygon with vertices in $\left\{\frac{\psi(\nu)}{l_{\nu}}:\nu\in\Omega_{\mathcal{S}}^{Per},l_{\nu}\leq M\right\}$.

\end{cor}

\end{subsection}

\end{section}

\begin{section}{Construction of the set $\mathcal{D}$.}\label{theset}

\ Let $M$ be a compact manifold and $f:M\rightarrow M$ be a diffeomorphism. We say that a compact set $\Lambda\subset M$ is a hyperbolic set if there exist $C>0$, $\lambda\in(0,1)$ and a decomposition of the tangent bundle over $\Lambda$ given by $T_xM=E^s_x\oplus E^u_x$ (stable and unstable vector bundles) such that:

\begin{itemize}

\item[(1)] $D_xf(E^s_x)=E^s_{f(x)},D_xf^{-1}(E^u_x)=E^u_{f^{-1}(x)}$ for every $x\in\Lambda$.

\item[(2)] $\lVert (D_xf)^n|_{E^s_x}\lVert<C\cdot\lambda^n$ and $\lVert (D_xf)^{-n}|_{E^u_x}\lVert<C\cdot\lambda^n$ for every $x\in\Lambda,n\in\N$.

\end{itemize}http://www.180.com.uy/

Further, we say that $\Lambda$ is a basic piece of $f$ if it is a compact invariant hyperbolic set, transitive ($\Lambda=cl[\{f^n(x):n\in\Z\}]$ for some $x\in\Lambda$) and locally maximal (there exist a neighborhood $U$ of $\Lambda$ such that $\bigcap_{n\in\Z}f^n(U)=\Lambda$).

Let $f:M\rightarrow M$ be a homeomorphism and $x\in M$, Then:

\begin{itemize}

\item for $\varepsilon>0$ the local stable set of size $\varepsilon$ is defined as
$$W^s_{\varepsilon}(x,f)=\{y\in M:d(f^n(y),f^n(x))<\varepsilon\mbox{ for every }n\in\N\},$$
and the local unstable set of size $\varepsilon$ as 
$$W^u_{\varepsilon}(x,f)=\{y\in M:d(f^{-n}(y),f^{-n}(x))<\varepsilon\mbox{ for every }n\in\N\};$$

\item the stable set is defined as 
$$W^s(x,f)=\{y\in M:\lim_{n\to +\infty}d(f^n(y),f^n(x))= 0\},$$ 
and the unstable set as 
$$W^u(x,f)=\{y\in M:\lim_{n\to -\infty}d(f^{-n}(y),f^{-n}(x))= 0\}.$$

\end{itemize}

For a hyperbolic set $\Lambda$ of some diffeomorphism $f$ we have the so called Stable Manifold Theorem.

\begin{teo}[Stable Manifold Theorem, \cite{Katok}]\label{smt}

Let $M$ be a compact manifold and $f:M\rightarrow M$ a $C^r$ diffeomorphism with hyperbolic set $\Lambda$. Then there exists $\varepsilon>0$ such that for every point $x\in\Lambda$ we have:

\begin{itemize}

\item[(1)] $W^s_{\varepsilon}(x,f)$ is a $C^r$ embedded manifold such that $T_x W^s_{\varepsilon}(x,f)=E^s_x$ and $W^s_{\varepsilon}(x,f)$ varies continuously with $x$ (as $C^r$ sub-manifold);

\item[(2)] $W^s_{\varepsilon}(x,f)\subset W^s(x,f)$;

\item[(3)] $W^s(x,f)=\bigcup_{n\in\N}f^{-n}(W^s_{\varepsilon}(x,f))$ is a $C^r$ immersed manifold.

\end{itemize}
\end{teo}

The analogous result holds for the unstable set.

A first remark concerning this theorem is that for a basic piece the dimension of the stable and unstable bundles are constant and they vary continuously (this can be proved directly from the definition of basic piece). This implies that we can define the index of a basic piece as the dimension of its unstable vector bundle. When a basic piece has index 1 we say it is a saddle piece. Further properties for basic sets will be presented in the next section. Next we want to define axiom A diffeomorphisms. We denote by $\textrm{Per}(f)$ the set of periodic points of $f$ and recall that the non-wandering set of $f$ is defined as:

$$\Omega(f):=\{x\in M:\forall U\mbox{ neighborhood of $x$ }f^n(U)\cap U\neq\phi\mbox{ for some }n\in\N \}$$

Let $M$ be a compact manifold. We say that a diffeomorphism $f:M\rightarrow M$ is axiom A if

\begin{itemize}

\item $\Omega(f)=\textrm{cl}[Per(f)]$;

\item $\Omega(f)$ is hyperbolic.

\end{itemize}

The dynamical properties of these systems are well-understood. Further, under additional hypothesis these systems are \textit{robust} in different senses and play a crucial role in the theory of generic dynamics (see references below). We will make use of these facts for constructing the set $\mathcal{D}$.

\begin{teo}[Spectral decomposition theorem, \cite{new}]\label{sdt}

Let $f:M\rightarrow M$ be an axiom A diffeomorphsim. Then, there exist a finite number of pairwise disjoint basic pieces $\Lambda_1,...,\Lambda_n\subset M$ for $f$ such that $\Omega(f)=\Lambda_1\cup...\cup\Lambda_n$. Further, for each $i=1,...,n$ there exists a finite number of compact and pairwise disjoint sets $\Lambda_{i_0},...,\Lambda_{i_{k-1}}$ such that:

\begin{itemize}

\item for every $j=0,...,k-1$ $f(\Lambda_{i_j})=\Lambda_{i_{j+1}}$ $mod(k-1)$;

\item for any $x\in \Lambda_{i_j}$ the sets $W^u(x,f),W^s(x,f)$ are dense in $\Lambda_{i_j}$;

\item $f^{k}\mid_{\Lambda_{i_j}}$ is topologically mixing.

\end{itemize}

\end{teo}

We say that an axiom A diffeomorphism satisfies the \textit{transversality condition} if for any two points $x,y\in \Omega(f)$, we have that $W^u(x,f)$ intersects $W^s(y,f)$ transversally.

The following theorem shows the semi-stability of axiom A diffeomorphisms with the transversality condition (for any compact manifold) and is due to Z. Nitecki.

\begin{teo}\cite{Zni}

Let $f\in\homeo$ be an axiom A diffeomorphism with the transversality condition. Then, there exists a neighborhood $\mathcal{U}(f)\subset \homeo$ of $f$ such that any $g$ of $\mathcal{U}(f)$ is semi-conjugate to $f$ by  a semiconjugacy $h$ homotopic to the identity.

\end{teo}

We mention here that the fact of the map $h$ to be homotopic to the identity is not explicitly stated in Z. Nietecki's article, but is a well known fact that this kind of maps can be constructed by the \emph{Shadowing L\index{\footnote{}}emma} such that they are close to the identity, and therefore homotopic to the identity. The above theorem, together with the properties of the rotation set given in Section \ref{intro}, implies the following corollary:

\begin{cor}\label{corneigh}

Let $f\in\homeo$ be an axiom A diffeomorphism that satisfies the transversality condition. Then there exists an open neighborhood $\mathcal{U}(f)\subset \homeo$ of $f$ such that for every $g\in\mathcal{U}(f)$ we have $\rho(g)=\rho(f)$.

\end{cor}

We say that a diffeomorphism $f:M\rightarrow M$ is a fitted axiom A diffeomorphism if it is axiom A, has the transversality condition and $\Omega(f)$ is totally disconnected. For a given manifold $M$ we denote by $\mathcal{F}$ the family of fitted axiom A diffeomorphisms.

\begin{teo}\cite{SS}\label{dense}

Let $M$ be a compact manifold. Then $\mathcal{F}$ is dense in $\textrm{Homeo}(M)$ with the $C^0$ topology.

\end{teo}

A nice exposition of this result in surfaces can be found in \cite{Franks}. We now turn to the existence of Markov partitions for totally disconnected saddle basic pieces of $f\in\homeo$, with some extra properties needed for further applications in the following sections. The main result that we are going to provide can be found for instance in \cite{Be}, where Markov partitions are defined in a stronger way, verifying at once all the properties that we request. We should say that the general idea of associating Markov partitions to basic pieces in any dimension, is due to R. Bowen \cite{bow}. Next we introduce some notations and remarks which will be used not in this section, but in the last section of this article.

\ We say that a hyperbolic set $\Lambda$ has local product structure if there exists $\delta_0>0$ such that given $\varepsilon_0>0$ as in the (Un)Stable Manifold Theorem and any two points $x,y\in\Lambda$ with $d(x,y)<\delta_0$ we have that $W^u_{\varepsilon_0}(x,f)\cap W^s_{\varepsilon_0}(y,f)\subset\Lambda$. It turns out that for a hyperbolic set $\Lambda$, the existence of a local product structure is equivalent to be locally maximal. Hence every basic piece has local product structure.

\ Let $\mathcal{S}$ be a partition of a totally disconnected saddle basic piece $\Lambda$. We say that $S\in\mathcal{S}$ is an $us$-box if

\begin{itemize}

\item $\textrm{diam}(S)<\delta_0$, where $\delta_0$ is given by the local product structure;

\item there exists a homeomorphism $h:[0,1]^2\rightarrow S$ such that for every $x\in S\cap\Lambda$ we have that

$$W^u_{\varepsilon_0}(x,f)\cap S=h([0,1]\times\{b_x\}),\ W^s_{\varepsilon_0}(x,f)\cap S=h(\{a_x\}\times[0,1])$$ where $(a_x,b_x)=h^{-1}(x)$, and 

$$h(\{(0,0),(0,1),(1,0),(1,1)\})\subset \Lambda.$$

\end{itemize}

We say that $\mathcal{S}$ is a $us$-partition if every element in $\mathcal{S}$ is a $us$-box. When $\mathcal{S}$ is a also a Markov partition we say that $\mathcal{S}$ is a $us$-Markov partition. Let $S$ be a $us$-box of a partition of $\Lambda$. We say that $H\subset S$ is a horizontal rectangle if $H=h([0,1]\times [c,d])$. Analogously we say that $V\subset S$ is a vertical rectangle if $V=h([a,b]\times [0,1])$.

\ Given an $us$-partition $\mathcal{S}=\{S_0,...,S_N\}$ of $\Lambda$, we say that it has the intersection property if for every $i,j\in\{0,...,N\}$, $x\in\Lambda\cap S_i$, $y\in\Lambda\cap S_j$ and $n\in\N$ with $f^n(x)=y$ we have

\begin{itemize}

\item $V=[f^{-n}(S_j)\cap S_i]_x$ is a vertical rectangle;

\item $H=[f^n(S_i)\cap S_j]_y$ is a horizontal rectangle.

\end{itemize}

\begin{obs}\label{verrhor}

Let $\mathcal{S}$ be an $us$-partition with the intersection property and $i,j\in\{0,...,N\}$. Further assume that there exists $x\in S_i$ and $y\in S_j$ such that $f^n(x)=y$ for some $n\in\N$. Then the vertical rectangle $V=[f^{-n}(S_j)\cap S_i]_x$ and the horizontal rectangle $H=[f^{n}(S_i)\cap S_j]_y$ verify $f^n(V)=H$.

\end{obs}

We are now ready to state the theorem about Markov partitions associated to totally disconnected saddle basic pieces with the desired properties.

\begin{teo}\cite{Be}\label{partitions}

Let $\Lambda$ be a totally disconnected saddle basic piece of $f\in\homeo$. Then for every $\varepsilon>0$ there exists a partition $\mathcal{S}$ of $\Lambda$ such that:

\begin{itemize}

\item[(1)] $\mathcal{S}$ is a $us$-Markov partition for $f$ with the intersection property, and $d_{\mathcal{S}}<\varepsilon$;

\item[(2)] for every $n\in\N$, $f^n_*(\mathcal{S})$ is a Markov partition of $\Lambda$ for $f^n$.

\end{itemize}

\end{teo}

The point (2) in the last Theorem is not explicit in the reference as a property for the partitions, but it is a direct consequence of their construction (and in general a well-known fact). Now by making use of the remarks \ref{smallsize} and \ref{pushisrot} we have the following crucial consequence:

\begin{cor}\label{partitionscor}
 
Let $\Lambda$ be a totally disconnected saddle basic piece of $f\in\homeo$. Then there exists a partition $\mathcal{S}$ such that:

\begin{itemize}

\item[(1)] $\mathcal{S}$ is a rotational $us$-Markov partition for $f$ with the intersection property, with $d_{\mathcal{S}}$ arbitrarily small;

\item[(2)] for every $n\in\N$, $f^n_*(\mathcal{S})$ is a rotational Markov partition of $\Lambda$ for $f^n$.

\end{itemize}

\end{cor}

The following is an useful observation for computing the rotation set of an axiom A diffeomorphism.

\begin{obs}\label{computing}

Let $\Lambda_1,...,\Lambda_n$ be basic pieces of $f$ such that $\textrm{conv}(\rho_{\Lambda_i}(F))$ are rational polygons for $i=1,...,n$. Then, if $\Lambda=\Lambda_1\cup...\cup\Lambda_n$ we have that $\textrm{conv}(\rho_{\Lambda}(F))$ is a rational polygon. Furthermore, if $\Lambda=\Omega(f)$ then $\rho(F)$ is a rational polygon.

\end{obs}

This observation is easily proved by making use of the properties of the rotation set given in Section \ref{intro}. Now we are able to construct the set $\mathcal{D}$ in the following theorem.

\begin{teo}\label{TEO1}

There exists an open and dense set $\mathcal{D}\subset \homeo$ such that for any $f\in\mathcal{D}$, the rotation set $\rho(F)$ is a rational polygon.

\end{teo}

\begin{proof}

Let us consider $\mathcal{F}_0=\{f\in \homeo:f\mbox{ is a fitted axiom A diffeomorphism}\}$. Given an element $f\in\mathcal{F}$, Theorem \ref{partitionscor} implies that we can associate a rotational Markov partition to each basic piece $\Lambda$ of $f$. Therefore, Corollary \ref{polygon} implies that $\textrm{conv}(\rho_{\Lambda}(F))$ is a rational polygon for each basic piece $\Lambda$. Hence by the last observation, $\rho(F)$ is a rational polygon for every $f\in\mathcal{F}_0$.

\ Now, if for every $f\in\mathcal{F}_0$ we consider the neighborhood $\mathcal{U}(f)$ given by the corollary \ref{corneigh}, then we have that $\mathcal{D}=\bigcup_{f\in\mathcal{F}_0}\mathcal{U}(f)$ is an open and dense set in $\homeo$, and for every $g\in\mathcal{D}$ the set $\rho(G)$ is a rational polygon for any lift $G$ of $g$.

\end{proof}

\end{section}

\begin{section}{Extension to any axiom A diffeomorphism.}\label{Axasec}

The aim of this section is prove that the rotation set of any axiom A diffeomorphism in $\homeo$ is a rational polygon. From the last section, we know that this is true for any fitted axiom A diffeomorphism in $\homeo$. Thus, one first may ask about the existence or not of axiom A diffeomorphisms which are not fitted in $\homeo$. However, is not difficult to construct, by using the Plykin attractor, an axiom A diffeomorphism $f\in\homeo$ with a basic piece $\Lambda$ that is non-totally disconnected (see \cite{Katok}). Further, one can manage the construction in a way that $\Lambda$ is not contained in a fordward invariant topological disk, so we dont have $\rho_{\Lambda}(F)=\{v\}$ with $v$ rational.

\ In what follows we introduce some preliminary notions and remarks, which are folklore of hyperbolic dynamics, see for instance \cite{Katok}. We have that every hyperbolic sets is expansive, that is: given a hyperbolic set $\Lambda$ there exists a positive constant $\gamma$ (constant of expansiveness) such that for any two points $x,y\in\Lambda$ there exist $n\in\Z$ for which $d(f^n(x),f^n(y))>\gamma$.

Let $M$ be a Riemannian manifold, $f:M\rightarrow M$ a diffeomorphism and $\{x_n\}_{n\in\Z}\subset M$. We say that $\{x_n\}_{n\in\Z}$ is a $\delta$-pseudo orbit if $d(f(x_{n}),x_{n+1})<\delta$ for every $n\in\Z$. It is said to be $\varepsilon$-shadowed by the orbit of $y\in M$ if $d(f^n(y),x_n)<\varepsilon$ for every $n\in\Z$.

\begin{teo}{Shadowing Lemma}\label{shadow}

Let $M$ be a Riemannian manifold, $f:M\rightarrow M$ a diffeomorphism, $\Lambda\subset M$ a basic piece of $f$ and $\gamma>0$ a constant of expansiveness in $\Lambda$. Then, there exists a neighborhood $U(\Lambda)$ of $\Lambda$ such that for every positive number $\varepsilon<\frac{\gamma}{2}$, there exists $\delta>0$ for which every $\delta$-pseudo orbit is $\varepsilon$-shadowed by a unique orbit in $\Lambda$.

\end{teo}

We say that a basic piece $\Lambda$ is an attractor if there exists an open set $U$ such that $f(cl[U])\subset  U$ and $\Lambda=\bigcap_{n\in\N}f^{n}(U)$, and it is a repeller if it is an attractor for $f^{-1}$. We say that it is a non-trivial attractor (repeller) when it is not a periodic orbit.

\begin{obs}\label{wucont}

If $\Lambda$ is a non-trivial attractor for a diffeomorphism $f$, $\Lambda$ has index one and $W^{u}(x,f)\subset \Lambda$ for every $x\in\Lambda$. The analogous result holds for repellers.

\end{obs}

\begin{obs}
 
Let $\Lambda$ be a non-trivial attractor (repeller) for $f\in\homeo$. Then $\Lambda\neq \T^2$.

In fact, if $\Lambda=\T^2$ it is known that $f$ is conjugate to an Anosov diffeomorphism (see \cite{Katok}), which is not possible since $f$ is homotopic to the identity.

\end{obs}

\begin{obs}\label{size}

Let $\Lambda\neq \T^2$ be a non-trivial attractor with a fixed point $P_0$ of $f\in\homeo$. Then, there exists $r'>0$ such that for every positive number $r<r'$ we have $[B(P_0,r)\cap \Lambda]_{P_0}=W^u_r(P_0,f)$. An analogous result holds for repellers.

Otherwise, due to the local product structure one can observe that a stable arc is contained in $\Lambda$. This, due to Remark \ref{wucont} and the local product structure, would imply that $\Lambda$ is an open set. Hence $\Lambda$ is closed and open so $\Lambda=\T^2$ which is absurd (see \cite{Katok} for the rigorous arguments).

\end{obs}

\begin{obs}\label{topbasicpiece}

Let $\Lambda$ be a non-totally disconnected basic piece of a diffeomorphism $f$ in $\homeo$ which is not a fixed point. Then, using that $\Lambda\neq\T^2$, we have that either it is a non-trivial attractor or a non-trivial repeller.
Furthermore, if $\Lambda$ contains a fixed point then it is a connected set.

This is a well-known statment in hyperbolic dynamics and can be proved by similar arguments as the remark above.

\end{obs}

The following Lemma refers to the topological stability of hyperbolic fixed points of saddle type.

\begin{lema}\label{topstab}

Let $P_0$ be a hyperbolic fixed point of saddle type for $f$. Then there exists $\delta_0>0$ such that given two positives numbers $r\in(0,\delta_0)$ and $\eta>0$ there exsits $\xi>0$ with the following property: if $d_0(\tilde{f},f)<\xi$ then there is a compact and connected set $K$ that verifies

\begin{itemize}

\item $d_{\mathcal{H}}(K,W^u_r(P_0,f))<\eta$, and for some $n\in\N$ we have that $K=\tilde{f}^n(K_0)$ where
\\$K_0\subset \bigcap_{n\in\N}\tilde{f}^n(B(P_0,\eta))$;

\item there exists a $\tilde{f}$-invariant set $K'\subset B(P_0,\eta)$ such that $K'\cap K\neq \emptyset$. 

\end{itemize}

(see figure \ref{indices}).

%
%

\end{lema}

\small

\begin{figure}[ht]\begin{center}

\psfrag{P0}{$P_0$}\psfrag{K}{$K$}\psfrag{K'}{$K'$}
\psfrag{WUr}{$W^u_r(P_0)$}\psfrag{WS}{$W^s(P_0)$}\psfrag{B}{$\phi(B)$}
\psfrag{Bola}{$B(W^u_r(P_0),\eta)$}

\includegraphics[height=7cm]{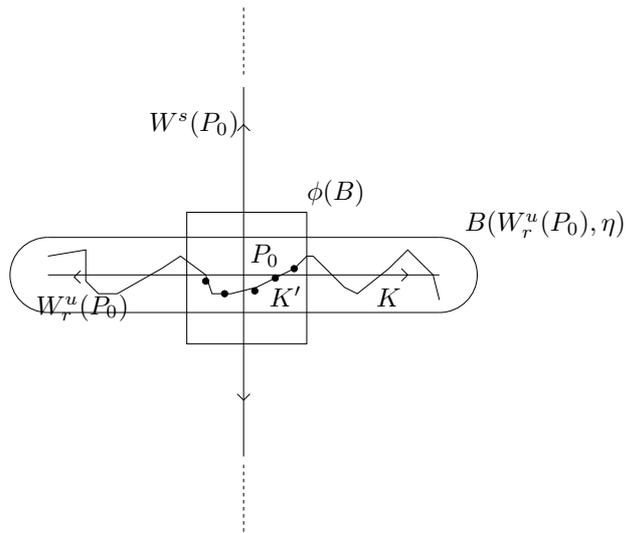}
\caption{Lemma \ref{topstab}. The set $\phi(B)$ is constructed in the proof.}\label{indices}
\end{center}\end{figure}

\normalsize

\begin{proof}

Let $\phi:\R^2\rightarrow U$ be the Hartman-Grobman coordinates with $\phi([-1,1]\times\{0\})=W^u_{\delta_0}(P_0,f)$ and $\phi(\{0\}\times [-1,1])=W^s_{\delta_0}(P_0,f)$. For positive numbers $r<\delta_0$ and $\eta$ let us fix a box $B=[a,b]\times[c,d]$, such that:

\begin{itemize}

\item[(i)] $0\in \textrm{int}(B)$, $\phi(B)\subset B(P_0,\eta)\cap U$;

\item[(ii)] $f(\phi(B))\subset U$;

\item[(iii)] there exists $n\in\N$ such that for any compact and connected set $L\subset B$ which intersect the sets $A_1=\{a\}\times [c,d]$ and $A_3=\{b\}\times [c,d]$, we have $d_{\mathcal{H}}(\phi((Df)_{P_0}^n(L)),W^u_r(P_0,f))<\eta$. 

\end{itemize}

Is not difficult to see that this box can be constructed. In what follows we say that a compact and connected set $L\subset B$ joins $A_1$ and $A_3$ if it intersects $A_1$ and $A_3$. Furhter, we use the analogous definitions for the sets $A_2=[a,b]\times \{c\}$ and $A_4=[a,b]\times\{d\}$.

Now, we consider $\xi_1>0$ such that if $d_0(\tilde{f},f)<\xi_1$ then $\tilde{g}=\phi^{-1}\circ \tilde{f}\circ \phi$ verifies:

\begin{itemize}

\item[(a)] for every compact and connected set $C\subset B$ which joins $A_1$ and $A_3$, the set $\tilde{g}(C)\cap B$ contains at least one connected component which joins $A_1$ and $A_3$;  

\item[(b)] for every compact and connected set $C\subset B$ which joins $A_2$ and $A_4$, the set $\tilde{g}^{-1}(C)\cap B$ contains at least one connected component which joins $A_2$ and $A_4$;

\end{itemize}

Therefore property (a) allows us to do the following inductive definition:

$K_0=B$, $\mathcal{F}_0=\{C:\mbox{ $C$ connected component of $\tilde{g}(K_0)\cap B$ which joins $A_1$ and $A_2$}\}$;

$K_1=\bigcup_{C\in\mathcal{F}_0}C$, $\mathcal{F}_0=\{C:\mbox{ $C$ connected component of $\tilde{g}(K_1)\cap B$ which joins $A_1$ and $A_2$}\}$;

\ and in general: 

$K_n=\bigcup_{C\in\mathcal{F}_{n-1}}C$, $\mathcal{F}_n=\{C:\mbox{ $C$ connected component of $\tilde{g}(K_{n})\cap B$ which joins $A_1$ and $A_2$}\}$;

\ We remark here that due to consider $C^0$-perturbations of $f$, we can not ensure that $K_n$ has a unique element for every $n\in\N$. Now, we define $K_I=\bigcap_{n\in\N}K_n$ which is a non-empty compact set, since it is a decreasing intersection of non-empty compact set (notice that the families $\mathcal{F}_n$ are finite). This implies that $B^+=\bigcap_{n\in\N} \tilde{g}^n(B)$ is not empty since $K_I\subset B^+$. Moreover if we pick a point $x\in K_I\cap B^+$, we have that $[K_I]_x$ joins $A_1$ and $A_3$.

In an analogous way by using the property (b), we can construct a non-empty compact set $K_{II}$ with $K_{II}\subset B^-:=\bigcap_{n\in\N}\tilde{g}^{-n}(B)$. Furthermore if we pick now $y\in K_{II}$, we have that $[K_{II}]_y$ joins $A_2$ and $A_4$. This implies that the set $\tilde{K}':= B^+\cap B^-=:\tilde{K}'$ is a non-empty compact and $\tilde{g}$-invariant set. Let $z\in\tilde{K}'$ and $\tilde{K}_0:=[K_I]_z\subset B^+$.

Due to the property (iii) in the definition of $B$, we can suppose the existence of $\xi<\xi_1$ such that if $d_0(\tilde{f},f)<\xi$ then for $K:=\tilde{f}^n(\phi(\tilde{K}_0))$ we have $d_{\mathcal{H}}(K,W^u_r(P_0,f))<\eta$. Furthermore, if we set $\phi(\tilde{K}_0)=:K_0$ we have by definition that $K=\tilde{f}^n(K_0)$, and by construction that $K_0\subset \bigcap_{n\in\N}\tilde{f}^n(B(P_0,\eta))$. Finally, if we take $K':=\phi(\tilde{K}')$ we have that $K'\subset B(P_0,\eta)$ and $K\cap K'\neq \emptyset$ since $\tilde{f}^n(\phi(z))\in K\cap K'$.

\end{proof}

The following theorem is the key result for our purpose. It relates, in a dynamical way, totally disconnected basic pieces with connected ones.

\begin{teo}\label{thm1}

Let $\Lambda$ be a non-trivial attractor of a diffeomorphism $f$ in $\homeo$, $P_0$ a fixed point in $\Lambda$, $U(\Lambda)$ as in Theorem \ref{shadow} and $U\subset U(\Lambda)$ an open and connected set such that $f(cl[U])\subset U$ and $\bigcap_{n\in\Z}f^n(U)=\Lambda$. Then, there exists $\varepsilon_0>0$ such that if $d_0(\tilde{f},f)<\varepsilon_0$ we have:

\begin{itemize}

\item[(1)] $\tilde{f}(cl[U])\subset U$ and $\tilde{\Lambda}=\bigcap_{n\in\N} \tilde{f}^n(U)$ is a non-empty compact and connected set;

\item[(2)] there exists a continuous and onto map $h:\Lambda\rightarrow \tilde{\Lambda}$ such that $d_0(h,\textrm{Id})<\frac{1}{4}$ and $h\circ \tilde{f}=f\circ h$;

\item[(3)] $h$ lifts to a continuous and onto map $H:\pi^{-1}(\tilde{\Lambda})\rightarrow \pi^{-1}(\Lambda)$ such that for any two lifts $\tilde{F},F$ of $\tilde{f},f$ we have $H\circ \tilde{F}=F\circ H+v$ for some $v\in\Z^2$, and $d_0(H,id)<\frac{1}{4}$.

\end{itemize}

\end{teo}

\begin{proof}

It is easy to see that there exists $\varepsilon'>0$ such that for all $\tilde{f}\in\homeo$ with $d_0(\tilde{f},f)<\varepsilon'$ statement (1) holds. In the following we want to achieve point (2) in the theorem in two steps. We first find $\varepsilon_1>0$ for which the existence and continuity of the map $h$ is guaranteed for any $\tilde{f}\in\homeo$ with $d_0(\tilde{f},f)<\varepsilon_1$ (i). The second step is find a positive number $\varepsilon_2\leq\varepsilon_1$, such that if $d_0(\tilde{f},f)<\varepsilon_2$ then the map $h$ is onto (ii).

\begin{itemize}

\item[i)] Let $\alpha>0$ be such that $\alpha<\textrm{min}\{\frac{\gamma}{2},\frac{1}{4}\}$, where $\gamma$ is the constant of expansiveness of $f$ in $\Lambda$, and consider $\delta_{\alpha}>0$ associated to $\alpha$ by Theorem \ref{shadow}. Consider now a positive number $\varepsilon_1<\min\{\delta_{\alpha},\varepsilon'\}$. Then for every $\tilde{f}\in\homeo$ with $d_0(\tilde{f},f)<\varepsilon_1$ we have that $\tilde{\Lambda}$ is a non-empty invariant set of $\tilde{f}$ contained in $U$ and every orbit of $\tilde{f}$ in $\tilde{\Lambda}$ is a $\delta_{\alpha}$-pseudo orbit of $f$. Hence, by Theorem \ref{shadow} to any point $x\in\tilde{\Lambda}$ we can associate a unique point $y_x\in B(x,\alpha)\cap \Lambda$ such that its orbit by $f$ $\alpha$-shadows the orbit of $x$ by $\tilde{f}$. Thus, we define $h:\tilde{\Lambda}\rightarrow \Lambda$ by $h(x)=y_x$.

\ Now we want to see that $h$ is a continuous function. Consider a sequence $\{x_n\}_{n\in\N}$ in $\tilde{\Lambda}$ converging to $x_0$, and $\{y_n=h(x_n)\}_{n\in\N}$. Suppose for a contradiction that $\{y_n\}_{n\in\N}$ does not converge to $h(x_0)=y_0$. Then we can consider a sequence $\{x_{n_k}\}_{k\in\N}$ converging to $x_0$ such that $\{y_{n_k}\}_{k\in\N}$ converges to $z_0\neq y_0$. It is not difficult to see that this implies $d(f^n(z_0),\tilde{f}^n(x_0))\leq \frac{\gamma}{2}$ for very $n\in\Z$. On the other hand by definition of $h$ $d(f^n(y_0),\tilde{f}^n(x_0))\leq \frac{\gamma}{2}$ for very $n\in\Z$. Therefore $d(f^n(z_0),f^n(y_0))\leq \gamma$ for very $n\in\Z$ which is absurd since $\gamma$ is the constant of expansiveness of $f$ in $\Lambda$. So we have for every $\tilde{f}\in\homeo$ with $d_0(\tilde{f},f)<\varepsilon_1$ a continuous map $h:\tilde{\Lambda}\rightarrow\Lambda$ such that $d_0(h,id)<\frac{1}{4}$ and $h\circ \tilde{f}=f\circ h$. Moreover, we have that $d_0(h,id)\leq d_0(\tilde{f},f)$.

\item[ii)] Now we want to find a positive number $\varepsilon_2\leq\varepsilon_1$, such that for every $\tilde{f}\in\homeo$ with $d_0(\tilde{f},f)<\varepsilon_2$ the map $h$ is onto. Consider $\delta_0>0$ given by Lemma \ref{topstab} for $P_0$. Define a positive number $r_0<\min\{\delta_0,\frac{\gamma}{2},r'\}$, where $r'$ is given by Remark \ref{size}, such that $B(P_0,r_0)\subset U$. Let $W=W^u_{r_0}(P_0,f)$, $r_1=\frac{r_0}{2}$, and $W'=W^u_{r_1}(P_0,f)$. Given a parametrization with some orientation for $W$ we can make use of standard notation for defining arcs $[c,d]$ in $W$. In particular $W$ itself is an arc $[a,b]$ such that $P_0\in(a,b)$ and $W'$ is an arc $[a',b']$ with $a<a'<P_0<b'<b$ (see figure \ref{kt1}).

\small

\begin{figure}[ht]\begin{center}

\psfrag{P0}{$P_0$}\psfrag{W'}{$W'$}
\psfrag{Wur0}{$W$}\psfrag{ws}{$W^s(P_0)$}\psfrag{a}{$a$}\psfrag{b}{$b$}\psfrag{a'}{$a'$}\psfrag{b'}{$b'$}
\psfrag{Br0}{$B(P_0,r_0)$}

\includegraphics[height=7cm]{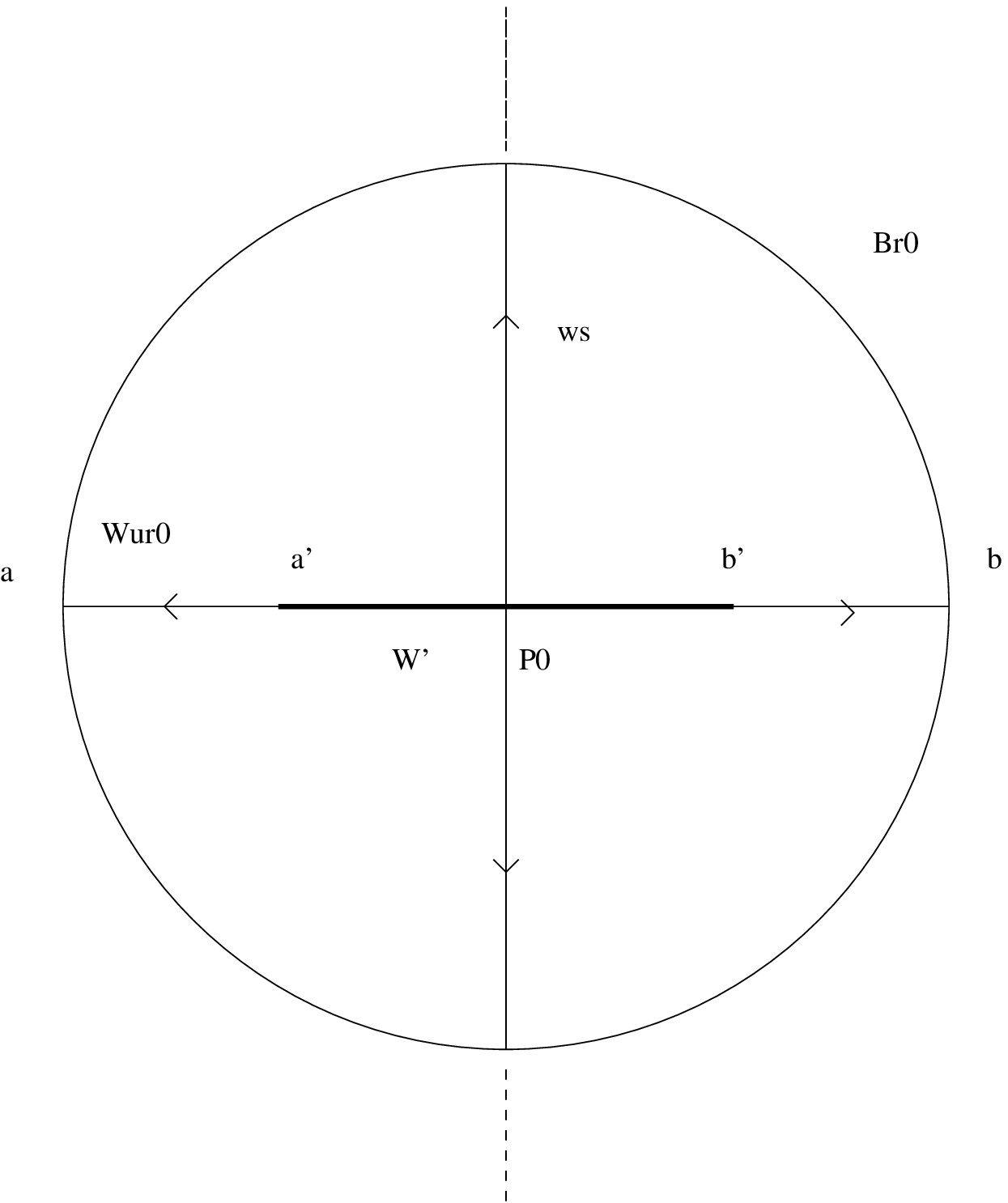}
\caption{}\label{kt1}
\end{center}\end{figure}

\normalsize

Now, let us consider $\eta>0$ such that
$$\min\{d(a,a'),d(b',b),d(a',[P_0,b]),d([a,P_0],b'),d(W',\partial B(P_0,r_0))\}>2\eta$$

and take $\xi'>0$ associated to $r_1$ and $\eta$ by Lemma \ref{topstab}. Thus, if we consider $\xi=\min\{\xi',\varepsilon_1\}$ by Lemma \ref{topstab} we have for any $\tilde{f}\in\homeo$ with $d_0(\tilde{f},f)<\xi$, the existence of a non-empty compact and connected set $K$ such that:

\begin{itemize}

\item $K$ is contained in $\tilde{\Lambda}$, since $K=\tilde{f}^n(K_0)$ with $K_0\subset \bigcap_{n\in\N}\tilde{f}^n(B(P_0,\eta))$;

\item $K$ intersects a compact and $\tilde{f}$-invariant set $K'$ with $K'\subset B(P_0\eta)$. 

\end{itemize}

(see figure \ref{kt2})

\small
\begin{figure}[ht]\begin{center}

\psfrag{P0}{$P_0$}\psfrag{W'}{$W'$}
\psfrag{Wur0}{$W$}\psfrag{ws}{$W^s(P_0)$}\psfrag{a}{$a$}\psfrag{b}{$b$}\psfrag{a'}{$a'$}\psfrag{b'}{$b'$}
\psfrag{Br0}{$B(P_0,r_0)$}\psfrag{BW'}{$B(W',\eta)$}\psfrag{K}{$K$}

\includegraphics[height=7cm]{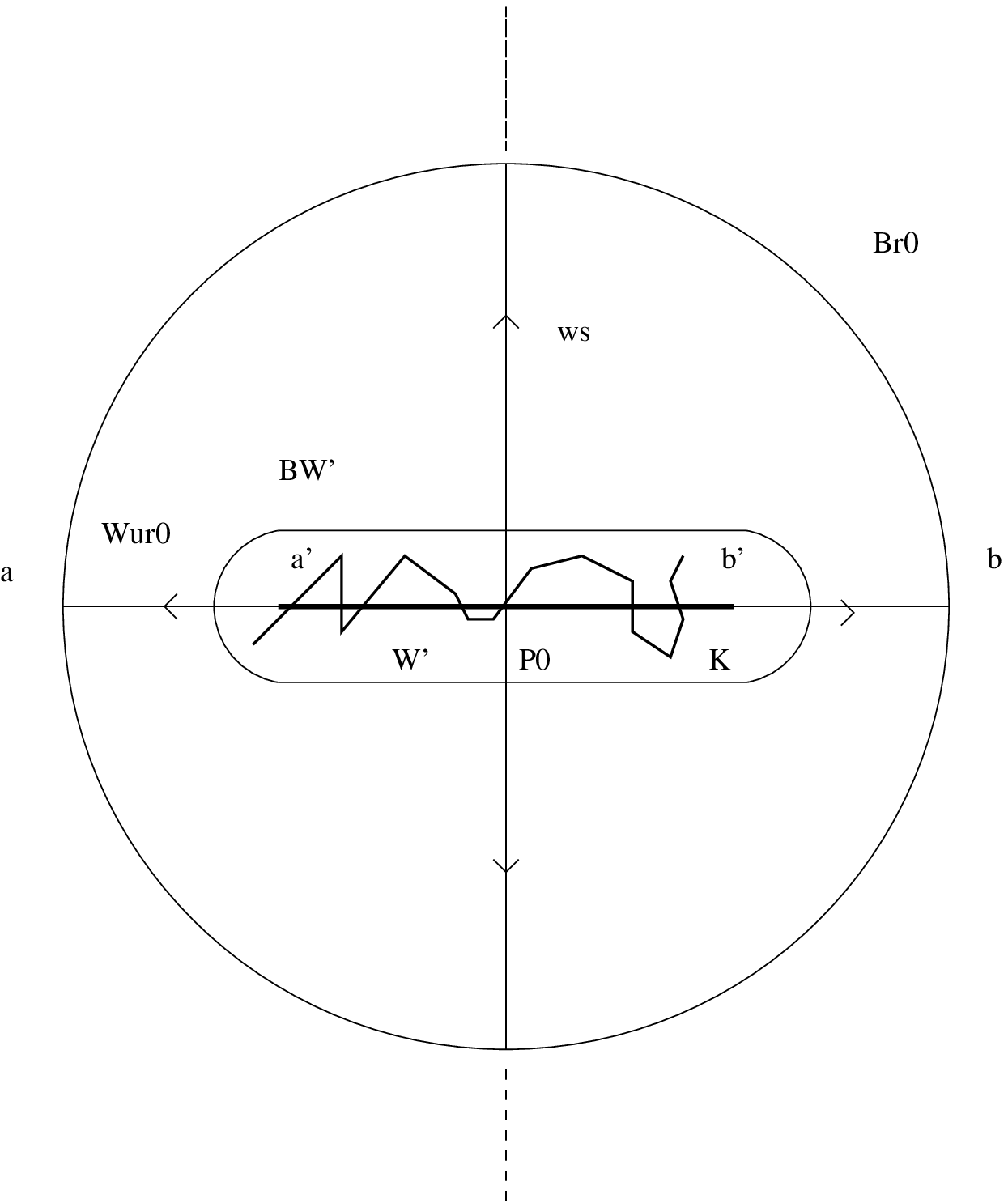}
\caption{}\label{kt2}
\end{center}\end{figure}

\normalsize

\ Due to the fact that $r_0<\frac{\gamma}{2}$ we have necessarily that $h(z)=P_0$ for every $z\in K'$, thus $h(K)$ contains $P_0$. Let $\varepsilon_2=\min\{\xi,\eta\}$. Then for every $\tilde{f}\in\homeo$ with $d_0(\tilde{f},f)<\varepsilon_2$ we have that $h(K)\subset B(W',2\eta)\subset B(P_0,r_0)$ is a compact and connected set containing $P_0$. Thus, by Remark \ref{size} $h(K)\subset W$. Furthermore if we choose $x'\in K\cap B(a',\eta)$ and $y'\in K\cap B(b',\eta)$, then for any $\tilde{f}\in\homeo$ with $d_0(\tilde{f},f)<\varepsilon_2$, we have that $d(h(x'),a')<2\eta$ and $d(h(y'),b')<2\eta$. Hence by definition of $\eta$ we have that $c:=h(x')\in [a,P_0)$ and $d:=h(y')\in (P_0,b]$. By connectedness of $K$ and continuity of $h$ we have that $[c,d]\subset h(K)$.

\ Due to the properties of the semi-conjugacy, the set $D=\bigcup_{n\in\N} f^n([c,d])$ is contained in $h(\tilde{\Lambda})$. On the other hand $D=W^u(P_0,f)$, which is dense in $\Lambda$ (see \ref{sdt}). Thus $h(\tilde{\Lambda})=\Lambda$ which means that $h$ is onto for every $\tilde{f}\in\homeo$ with $d_0(\tilde{f},f)<\varepsilon_2$.

\end{itemize}

We have already shown the existence of $\varepsilon_0=\textrm{min}\{\varepsilon',\varepsilon_1,\varepsilon_2\}$ such that for any $\tilde{f}\in\homeo$ with $d_0(\tilde{f},f)<\varepsilon_0$ (1) and (2) in the theorem are verified. Finally, we want to see that if $d_0(\tilde{f},f)<\varepsilon_0$ then the map $h$ lift to a map $H$ with the properties given in the point (3).

For this, consider a map $\tilde{f}\in\homeo$ with $d_0(\tilde{f},f)<\varepsilon_0$ and for $\tilde{x}\in\pi^{-1}(\tilde{\Lambda})$ the point $h(\pi(\tilde{x}))=y\in\Lambda$. Then since $d_0(h,Id)<\frac{1}{4}$, there is a unique point in $B(\tilde{x},\frac{1}{4})\cap \pi^{-1}(y)$ that we call $\tilde{y}_x$. We define $H:\pi^{-1}(\tilde{\Lambda})\rightarrow \pi^{-1}(\Lambda)$ by $H(\tilde{x})=\tilde{y}_x$. It is not difficult to see that this function is continuous, onto, and $d_0(H,Id)<\frac{1}{4}$. Furthermore if we consider $x\in\pi^{-1}(\tilde{\Lambda})$ and $v\in\Z^2$, then $H(x+v)=H(x)+v$. So $H$ lifts $h$.

Further, for fixed lifts $\tilde{F},F$ of $\tilde{f},f$ and any point $x\in\pi^{-1}(\tilde{\Lambda})$ we have by the definition of $H$ that $H\circ \tilde{F}(x)=F\circ H(x)+v(x)$ with $v(x)\in\Z^2$. Therefore we have that $v(x)=H\circ \tilde{F}(x)-F\circ H(x)$ defines a continuous and integer-valued function in $\T^2$ over the set $\tilde{\Lambda}$, which is compact and connected. Hence there exists $v\in\Z^2$ such that $v(x)=v$ for every $x\in\pi^{-1}(\tilde{\Lambda})$. This completes the proof.

\end{proof}

In the following we make use of the above theorem to compute the rotation set on basic pieces which are either a non-trivial attractor or a non-trivial repeller.

\begin{teo}\label{thm2}

Let $f\in\homeo$ be a diffeomorphism, $\Lambda$ be a non-trivial attractor of $f$ and $P_0$ a fixed point in $\Lambda$. Then $\textrm{conv}(\rho_{\Lambda}(F))$ is a rational polygon. An analogous result holds for repellers.

\end{teo}

\begin{proof}

Let $U$ be an open and connected set such that $f(cl[U])\subset U$ and $\bigcap_{n\in\Z}f^n(U)=\Lambda$. Let $\varepsilon_0>0$ be given by the Theorem \ref{thm1}. Due to Theorem \ref{dense} we can consider a fitted axiom A diffeomorphism $\tilde{f}\in\homeo$ that $d_0(\tilde{f},f)<\varepsilon_0$. Therefore we have:

\begin{itemize}

\item $\tilde{f}(cl[U])\subset U$ and $\tilde{\Lambda}=\bigcap_{n\in\Z}\tilde{f}(U)$ is a non empty compact and connected set;

\item there exists a continuous and onto $H:\pi^{-1}(\tilde{\Lambda})\rightarrow \pi^{-1}(\Lambda)$ such that for any two lifts $\tilde{F},F$ of $\tilde{f},f$ we have $H\circ \tilde{F}=F\circ H+v$ for some $v\in\Z^2$, and $d_0(H,id)<\frac{1}{4}$. Further $H(x)+v=H(x+v)$ for every $x\in \pi^{-1}(\tilde{\Lambda}),v\in\Z^2$.

\end{itemize}

This implies on one hand that for every $x\in\pi^{-1}(\tilde{\Lambda})$ if $y=H(x)$ then $H\circ \tilde{F}^n(x)=F^n(y)+v$. Therefore, for every  $x\in\pi^{-1}(\tilde{\Lambda})$ there exists $y\in\pi^{-1}(\Lambda)$ such that $d(F^n(y),\tilde{F}^n(x))<\frac{1}{4}+\| v \|$ for every $n\in\Z$. On the other hand for every $y\in \pi^{-1}(\Lambda)$ one can consider $x\in H^{-1}(y)$ and have that $H\circ \tilde{F}^n(x)=F^n(y)+v$, so we have for every $y\in\pi^{-1}(\Lambda)$ a point $x\in\pi^{-1}(\tilde{\Lambda})$ such that $d(F^n(y),\tilde{F}^n(x))<\frac{1}{4}+\| v\|$ for every $n\in\Z$.

Is not difficult to see that the last observation implies $\rho_{\Lambda}(F)=\rho_{\tilde{\Lambda}}(\tilde{F})$. Due to the fact that $\tilde{f}(cl[U])\subset U$, and that $\tilde{f}$ is a fitted axiom A diffeomorphism, we have that $\tilde{\Lambda}\cap \Omega(\tilde{f})=\Lambda_1\cup...\cup \Lambda_m$ with $\Lambda_i$ totally disconnected basic pieces of $\tilde{f}$ for $i=1,...,m$. This implies by Remark \ref{computing} that $\textrm{conv}(\rho_{\tilde{\Lambda}}(\tilde{F}))$ is a rational polygon, so it is $\textrm{conv}(\rho_{\Lambda}(F))$.

\end{proof}

\begin{cor}

Let $f\in\homeo$ be an axiom A diffeomorphism. Then $\rho(F)$ is a rational polygon.

\end{cor}

\begin{proof}

Due to Theorem \ref{sdt} there exists $n\in\N$ for which $f^n$ is an axiom A diffeomorphism with basic pieces $\Lambda_1,...,\Lambda_m$, such that for each $i=1,...,m$ the piece $\Lambda_i$ contains a fixed point $P_i$ of $f^m$. Whenever $\Lambda_i$ is a totally disconnected basic piece of $f$ we have seen in section \ref{theset} that $\textrm{conv}(\rho_{\Lambda_i}(F))$ is a rational polygon. On the other hand when it is not totally disconnected due to Remark \ref{topbasicpiece} it is either a fixed point or a non-trivial attractor or a non-trivial repeller. Thus, by Theorem \ref{thm2} we have that $\textrm{conv}(\rho_{\Lambda_i}(F^m))$ is a rational polygon in any case. Therefore Remark \ref{computing} implies that $\rho(F^m)$ is a rational polygon, and since $\rho(F)=\frac{1}{m}\rho(F^m)$ we have that $\rho(F)$ is a rational polygon.

\end{proof}

\end{section}

\begin{section}{Minimal sets with rotation set given by a non trivial continuum.}

In this section we observe the existence of minimal sets $\mathcal{M}\subset \T^2$ for some $f\in\homeo$, such that $\rho_{\mathcal{M}}(F)$ is a non trivial segment. We also observe that this phenomena is abundant in $\homeo$ (it occurs on an open set of the space). We start by proving that the rotation sets associated to minimal sets are always continua (compact and connected sets).

Given $F$ lift of $f\in\homeo$ we define for any $n\in\N$ the function $\varphi_n:\T^2\rightarrow \T^2$ by $\varphi_n(x):=\frac{F^n(x)-x}{n}$.

\begin{obs}\label{desplacement}

Let $f\in\homeo$. Then there exists $K>0$ such that $\parallel \varphi_{n+1}(x)-\varphi_n(x) \parallel<\frac{K}{n}$ for every $x\in\R^2$, $F$ lift of $f$ and $n\in\N$.

\end{obs}

\begin{teo}\label{rotminsetiscont}

Let $f\in\homeo$ and assume that $\mathcal{M}$ is a minimal set of $f$ and $F$ is a lift of $f$. Then $\rho_{\mathcal{M}}(F)$ is a continuum.

\end{teo}

\begin{proof}

Suppose for a contradiction that $\rho_{\mathcal{M}}(F)=A_1\cup A_2$, where $A_1,A_2$ are non-empty disjoint closed sets. Let $\delta=\frac{1}{3}d(A_1,A_2)$. Definition of the rotation set implies that there exists $n_0\in\N$ such that $\varphi_n(\mathcal{M})\subset B(A_1,\delta)\cup B(A_2,\delta)$ for every $n\geq n_0$. Due to Remark \ref{desplacement} we can assume at the same time that $\| \varphi_{n+1}(x)-\varphi_{n}(x) \|<\delta$ for all $n\geq n_0$ and $x\in\mathcal{M}$. As a consequence, $\varphi_{n_0}(x)\in B(A_i,\delta)$ implies $\varphi_{n}(x)\in B(A_i,\delta)$ for every $n\geq n_0$, $x\in\mathcal{M}$ and $i=1,2$. This directly yields that $\mathcal{M}$ decomposes into two disjoint sets:

$$ C_i=\{x\in\mathcal{M}:\rho_{x}(F)\subset A_i\}=\{x\in\mathcal{M}:\varphi_{n_0}(x)\in \textrm{cl}[B_{\delta}(A_i)]\};i=1,2$$ 

From the first characterization of $C_i$ we see that they are $f$-invariant, whereas the second characterization implies that they are closed. Further none of them can be empty, since this would mean that $\mathcal{M}=C_i$ for some $C_i$, and therefore $\varphi_{n}(\mathcal{M})\subset \textrm{cl}[B(A_i,\delta)]$ for all $n>n_0$, which implies that $\rho_{\mathcal{M}}(F)=A_i$ in contradiction to the definition of $A_1$ and $A_2$.

Hence, we have decomposed $\mathcal{M}$ in two non-empty closed and $f$-invariant sets, contradicting minimality.

\end{proof}

Let us now consider a space of sequences of finite symbols $\Omega=\{0,...,N\}^{\Z}$. Let $\xi':\{0,...,n_1\}\rightarrow \Omega$ be a finite word and $\xi:\{0,...,n_2\}\rightarrow \Omega$ either be a finite or infinite ($n_2=\infty$) word, with $n_2\geq n_1$. We say that $\xi'$ is a sub-word of $\xi$ if there exist $t\in\N,0\leq t\leq n_2-n_1$ such that $\xi(i+t)=\xi'(i)$ for every $i=0,...,n_1$. We say that $\xi\in\Omega$ is an almost periodic word if for every sub-word $\xi'$ of $\xi$ there exist $T_{\xi'}\in\N$, such that for every sub-word $\alpha$ of $\xi$ with length $T_{\xi'}$, we have that $\xi'$ is a sub-word of $\alpha$.

\begin{obs}\label{constminsets}

Let $\sigma:\Omega\rightarrow \Omega$ be the shift, that is $\sigma(\xi)(i)=\xi(i+1)$. Then for any almost periodic sequence $\xi\in\Omega$ the set $\mathcal{M}=cl[\{\sigma^n(\xi):n\in\Z\}]$ is a minimal set.

\end{obs}

In the following lemma some special almost periodic sequences are constructed. Given $\xi\in\Omega$ we define $S_n(\xi)=\frac{\sum_{i=0}^{n-1} \xi(i)}{n}$.

\begin{lema}[Communicated by Tobias J\"{a}ger]\label{tobias}

For every $\delta>0$ there exists an almost periodic sequence $\xi\in\{0,1\}^{\Z}$, such that $\liminf S_n(\xi)<\delta$ and $\limsup S_n(\xi)>1-\delta$.

\end{lema}

\begin{proof}

Let us fix $\delta>0$ and consider $t\in\N$ such that $2^{-t}<\delta$. We define recursively the sequence $\{a_n\}_{n\in\N}$ by $a_{n+1}=2^{t+n}a_n$ and $a_0=1$. Further we consider the following family of subsets of the positive integers:

\begin{itemize}

\item $A_n=[0,...,a_n]$, $B_0=A_0$;

\item $B_n=\left( A_n+a_{n+1}\N\right)\setminus B_{n-1}$ for every $n>0$. Note that $\biguplus_{n=0}^{\infty}B_n=\N$.

\end{itemize}
 
We consider now $\xi\in\Omega$ defined by

\begin{equation*}
\xi(i)= \ \begin{cases}

0:\mid i\mid\in B_n, \mbox{ $n$ even}\\
\\
1:\mid i\mid \in B_n, \mbox{ $n$ odd}

\end{cases}
\end{equation*}

Given a finite sub-word $\xi'$ of $\xi$ we have by construction a sub-word $w$ of $\xi$, which is the restriction of $\xi$ to some $A_{n_0}$, such that $\xi'$ is a sub-word of $w$. Furthermore again by construction, we have that every sub-word of $\xi$ with Length $a_{n_0}+2^{t+n_0}a_{n_0}$ contains $w$ as a sub-word. This implies the almost periodicity for $\xi$. 

Now, on the other hand, if we consider $n\in\N\setminus\{0\}$ even we have by construction that $S_{a_n}<\frac{1}{a_n}\cdot a_{n-1}$, and hence $S_{a_n}<\frac{1}{a_n}\cdot a_{n-1}<2^{-t}<\delta$. On the other hand, if we consider $n\in\N$ odd we have that $S_{a_n}>\frac{1}{a_n}(a_n-a_{n-1})>1-\frac{a_{n-1}}{a_n}>1-2^{-t}>1-\delta$. So, $\liminf S_n(\xi)<\delta$ and $\limsup S_n(\xi)>1-\delta$.

\end{proof}

An important corollary of the last lemma is the following:

\begin{cor}\label{exwildmin}

Let $C$ be an invariant set of $f\in\homeo$ which admits a rotational Markov partition $\mathcal{S}=\{S_0,S_1\}$. Further, let $F$ be a lift of $f$ and assume that $\tau(0)=\{0,1\}$, $\tau(1)=\{0,1\}$ and $s_0\neq s_1$. Then there exists a minimal set $\mathcal{M}$ of $f$ such that $\rho_{\mathcal{M}}(F)$ is a non-trivial segment.

\end{cor}

\begin{proof}

Let us first assume that $s_0=0$ and $s_1=(1,0)$. Consider $\delta\in (0,\frac{1}{2})$, and $\xi\in\Omega_{\mathcal{S}}$ given by Lemma \ref{tobias}. By Remark \ref{constminsets} we have that $\mathcal{M}'=cl[\{\sigma^n(\xi):n\in\Z\}]$ is a minimal set for the shift $\sigma:\Omega_{S}\rightarrow\Omega_{\mathcal{S}}$.

Let $\mathcal{M}$ be a minimal set of $f$ contained in $h_{\mathcal{S}}^{-1}(\mathcal{M}')$. Since $\mathcal{M}'$ is a minimal set of $\sigma$, we have that $h_{\mathcal{S}}(\mathcal{M})=\mathcal{M}'$, and in particular $\xi\in h_{\mathcal{S}}(\mathcal{M})$. Then, by defining
\\$\xi_n:\{0,...,n-1\}\rightarrow\{0,1\}$ as $\xi_n(i)=\xi(i)$ for every $i\in\Z$, we have that $\xi_n\in\Omega_{\mathcal{S}}^f$ and $\frac{\psi(\xi_n)}{n}=S_n$ for every $n\in\N$. Thus, by Corollary \ref{rs} $\liminf S_n(\xi)$ and $\limsup S_n(\xi)$ are two different points in $\rho_{\mathcal{M}}(F)$. Further, by construction we know that $\rho_{\mathcal{M}}(F)\subset \textrm{conv}(\{(0,0),(1,0)\})$. Then by Theorem \ref{rotminsetiscont} we have that $\rho_{\mathcal{M}}(F)$ is a non trivial segment.

Now, let $L:\R^2\rightarrow \R^2$ be a afine linear isomorphism with $L(\Z^2)\subset \Z^2$ and denote by $A_L$ the induced torus automorphism. Then if $C$ is a closed $f$-invariant set and $g:=A_L\circ f\circ A_L^{-1}$, we have $\rho_{A_L^{-1}(C)}(g)=L(\rho_C(f))$ (see \cite{Kw1}). Hence by considering $L$ taking $s_0$ to $0$ and $s_1$ to $(1,0)$ and applying the argument above to $g$, we get the result in full generality.

\end{proof}

In what follows, we want to see that the existence of minimal sets with non-trivial segments as rotation sets is abundant in the $C^0$ topology. In fact, we will prove first that for any axiom A diffeomorphism whose rotaion set has non-empty interior this kind of minimal sets exists, and then we extend its existence to an open set in $\homeo$. Given any basic piece $\Lambda\subset\T^2$ we write $\tilde{\Lambda}=\pi^{-1}(\Lambda)$.

\begin{lema}\label{prevexwimin}

Let $\Lambda$ be a totally disconnected saddle basic piece of $f\in\homeo$ with a fixed point $P_0\in\Lambda$. Further consider a rotational Markov partition $\mathcal{S}=\{S_0,...,S_N\}$ of $\Lambda$ such that $P_0\in S_0$ and assume that $\textrm{conv}(\rho_{\Lambda}(F))$ is not a singleton. Then, for a lift $F$ of $f$ and $\tilde{P}_0\in\pi^{-1}(P_0)$ with $F(\tilde{P}_0)=\tilde{P_0}$, we have that $W^u(\tilde{P}_0,F)\cap \tilde{\Lambda}$ has non-empty intersection with at least two connected components of $\pi^{-1}(S_0)$.

\end{lema}

\begin{proof}

To prove this it is sufficient to do it for some particular $\tilde{P}_0$ and $F$ with $F(\tilde{P}_0)=\tilde{P_0}$. Let $D$ be a trivial domain such that $\tilde{P}_0\in D$, $\pi(D)\supset \bigcup_{i=0}^{N}S_i$ and  let $\tilde{S}_i=D\cap \pi^{-1}(S_i)$. We affirm first that there exists $i_0\in\{0,...,N\}$ such that $W^u(\tilde{P}_0,F)\cap \tilde{\Lambda}\cap \tilde{S}_{i_0}+v_0\neq\emptyset$ and $W^u(\tilde{P}_0,F)\cap \tilde{\Lambda}\cap \tilde{S}_{i_0}+v_1\neq\emptyset$ for some $v_0,v_1\in\Z^2$ with $v_0\neq v_1$. Otherwise we have necessarily that $cl[W^u(\tilde{P}_0,F)]\cap\tilde{\Lambda}=:\tilde{\Lambda}_0$ is a bounded $F$-invariant set with $\pi(\tilde{\Lambda}_0)=\Lambda$. Hence $\rho_{\Lambda}(F)=\{0\}$ which is absurd.

Now, we can consider two points $\tilde{x}_0,\tilde{x}_1\in\tilde{S}_0\cap\tilde{\Lambda}$ such that $F^n(\tilde{x}_0)\in \tilde{S}_{i_0}+v_0$ and $F^n(\tilde{x}_1)\in \tilde{S}_{i_0}+v_1$. For $x_0=\pi(\tilde{x}_0)$ and $x_1=\pi(\tilde{x}_1)$ we consider the finite words 
\\$\alpha_0:\{0,...,n\}\rightarrow \{0,...,N\}$ and $\alpha_1:\{0,...,n\}\rightarrow \{0,...,N\}$ such that $\alpha_0(i)=h_{\mathcal{S}}(x_0)(i)$ and $\alpha_1(i)=h_{\mathcal{S}}(x_1)(i)$ for every $i\in \{0,...,N\}$. Further since $\Lambda$ is transitive we can consider a finite word $\beta:\{0,...,m\}\rightarrow \{0,...,N\}$ such that $\beta(0)=i_0$ and $\beta(m)=0$.

As we have seen in the section \ref{Simbolicstuff} we can take now $\xi_0,\xi_1\in\Omega_{\mathcal{S}}$ defined as:

\begin{equation*}
\xi_0(i)= \ \begin{cases}

0:i\leq0,i\geq n+m\\
\\
\alpha_0\beta(i):0\leq i\leq n+m

\end{cases}
\\ \\ \mbox{ , }\xi_1(i)= \ \begin{cases}

0:i\leq0,i\geq n+m\\
\\
\alpha_1\beta(i):0\leq i\leq n+m

\end{cases}
\end{equation*}

Let $y_0\in h_{\mathcal{S}}^{-1}(\xi_0)$, $y_1\in h_{\mathcal{S}}^{-1}(\xi_1)$, $\tilde{y}_0\in\pi^{-1}(y_0)\cap\tilde{S}_0$ and $\tilde{y}_1\in\pi^{-1}(y_1)\cap \tilde{S}_0$. Is easy to see that $\tilde{y}_0,\tilde{y}_1\in W^u(\tilde{P}_0,F)\cap \tilde{\Lambda}$.

Due to the fact that $\mathcal{S}$ is a rotational Markov partition, we have on the one hand that $F^i(\tilde{y}_0)$ and $F^i(\tilde{x}_0)$ belong to the same connected component of $\pi^{-1}(\bigcup_{i=0}^N S_i)$ for every $i=0,...,n$, and the same holds for $F^i(\tilde{y}_1)$ and $F^i(\tilde{x}_1)$. Otherwise we necessarily have that $F(S_{i*})$ intersects $S_{j*}+u_0$ and $S_{j*}+u_1$, for some $i_*,j_*\in\{0,...,N\}$ with $u_0\neq u_1\in \Z^2$. 

On the other hand, by the same reason, for every $j\in\{0,...,m\}$ we have the following:

\begin{center}

if $F^{n+j}(\tilde{y}_0)\in \tilde{S}_{i_j}+w_j$, $w_j\in\Z^2$, then $F^{n+j}(\tilde{y}_1)\in \tilde{S}_{i_j}+w_j+v$, where $v=v_1-v_0$.

\end{center}

Hence $F^{n+m}(\tilde{y}_0)\in \tilde{S}_{0}+w_{m}$ and $F^{n+m}(\tilde{y}_1)\in \tilde{S}_{0}+w_{m}+v$ with $v\neq 0$.

\end{proof}

\begin{teo}\label{wildminsetaxa}

Let $\Lambda$ be a totally disconnected saddle basic piece of $f\in\homeo$ with a fixed point $P_0\in\Lambda$. Further assume that $\textrm{conv}(\rho_{\Lambda}(F))$ is not a singleton for some $F$ lift of $f$. Then $\Lambda$ contains a minimal set $\mathcal{M}$ such that $\rho_{\mathcal{M}}(F)$ is a non-trivial segment.

\end{teo}

\begin{proof}

Let us first fix $\tilde{P_0}\in\pi^{-1}(P_0)$ and a lift $F$ of $f$ with $F(\tilde{P}_0)=\tilde{P}_0$. We consider a rotational $us$-Markov partition $\mathcal{S}=\{S_0,...,S_N\}$ given by Corollary \ref{partitionscor} with $P_0\in S_0$. Further we consider a trivial domain $D$ such that $\pi(D)\supset \bigcup_{i=0}^N S_i$ and $\tilde{S}_i=D\cap \pi^{-1}(S_1)$. Thus we know by Lemma \ref{prevexwimin} that there exists $n\in\N$ and $\tilde{x}\in\tilde{S}_0\cap\tilde{\Lambda}$ such that $F^n(\tilde{x})\in \tilde{S}_0+v\cap\tilde{\Lambda}$ with $v\in\Z^2\setminus\{0\}$. Set $\pi(\tilde{x})=x$.

We now define $g=f^n$. Due to the choice of $\mathcal{S}$ we know that $g_*(\mathcal{S})$ is a rotational
\\ $us$-Markov partition of $\Lambda$ for $g$. Furthermore the sets $S'_0:=[f^{-n}(S_0)\cap S_0]_{P_0}=$
\\ $\pi([F^{-n}(\tilde{S}_0)\cap \tilde{S}_0]_{\tilde{P}_0})$ and $S'_1:=[f^{-n}(S_0)\cap S_0]_{x}=\pi([F^{-n}(\tilde{S}_0)\cap \tilde{S}_0]_{\tilde{x}})$ are two elements in $g_*(\mathcal{S})$, which are vertical rectangles of $S_0$, and the sets $T'_0=f^n(S'_0)$, $T'_1=f^n(S'_1)$ are horizontal rectangles in $S_0$.

Due to the local product structure of $\Lambda$ and the fact that $\mathcal{S}$ is a $us$-partition, we have that
\\ $S'_0\cap T'_0\cap \Lambda\neq \phi$, $S'_0\cap T'_1\cap\Lambda\neq \phi$, $S'_1\cap T'_0\cap\Lambda\neq \phi$ and $S'_1\cap T'_1\cap \Lambda\neq \phi$. Hence we have that $\tau_g(0)\supset \{0,1\}$ and $\tau_g(1)\supset \{0,1\}$. Further by definition we have $s'_0=0$ and $s'_1=v$ for $F^n$.

Let us now consider $C=h^{-1}_{g_*(S)}(\{\xi\in\Omega_{g_*(\mathcal{S})}:\xi(i)\in\{0,1\}\mbox{ for every $i\in\Z$}\})$. The above argument implies that $C$ is a a non-empty $g$-invariant closed set. Furthermore by construction $\mathcal{S}'=\{S'_0,S'_1\}$ is a rotational Markov partition of $C$ for $g=f^n$ with $s'_0=0$ and $s'_1=v$ for $F^n$. Therefore due to Corollary \ref{exwildmin} there exists a minimal set $\mathcal{M}'$ of $g$ such that $\rho_{\mathcal{M}'}(F^n)$ is a non-trivial segment. Hence, defining $\mathcal{M}=\mathcal{M}'\cup...\cup f^{n-1}(\mathcal{M'})$ we have that $\mathcal{M}$ is a minimal set of $f$ contained in $\Lambda$ with $\rho_{\mathcal{M}}(F)=\frac{1}{n}\rho_{\mathcal{M}'}(F^n)$.

\end{proof}

We have thus obtained a criterion to ensure the existence of minimal sets with non-trivial segments as rotation sets. Using this result we can now show that this phenomena is abundant in $\homeo$.

\begin{cor}\label{abun0}

Let $f\in\homeo$ be a fitted axiom A diffeomorphism with $\textrm{int}(\rho(F))$ non-empty. Then, there exists a minimal set $\mathcal{M}$ of $f$ such that $\rho_{\mathcal{M}}(F)$ is a non-trivial segment.

\end{cor}

\begin{proof}

By the same argument as in the final part of the previous proof, is sufficient to show the existence of a minimal set with the rotation set given by a non-trivial segment for some power of $f$. Let us consider $n\in\N$ such that each basic piece of $g:=f^n$ has a fixed point of $g$. Due to a result by Franks \cite{F2}, we have for any vector $v\in\Q^2\cap \textrm{int}(\rho(G))$ a periodic point $x\in\T^2$ such that $\lim_n \frac{G^n(\tilde{x})-\tilde{x}}{n}=v$, where $G=F^n$. Therefore, if $\Lambda_1,...,\Lambda_m$ are the basic pieces of $g$, there exists $i_0\in\{1,...,m\}$ such that $\Lambda_{i_0}$ contains two periodic points $x_1,x_2$ with $\lim_n \frac{G^n(\tilde{x}_1)-\tilde{x}_1}{n}=v_1$ and $\lim_n \frac{G^n(\tilde{x}_2)-\tilde{x}_2}{n}=v_2$ for two different vectors $v_1$ and $v_2$. Thus $\textrm{conv}(\rho_{\Lambda}(G))$ is not a singleton.

Further, it is not difficult to see that $\Lambda$ must be a totally disconnected saddle basic piece, which by definition of $g$ contains some fixed point. Thus, by Theorem \ref{wildminsetaxa} we obtain the existence of a minimal set of $g$ with a non-trivial segment as rotation set.

\end{proof}

By the same techniques as in the proof of Theorem \ref{thm2}, one can achieve the generalization of the last corollary to any axiom A diffeomorphism in $\homeo$.

\begin{cor}

Let $f\homeo$ be a axiom A diffeomorphism with $\textrm{int}(\rho(F))$ non empty for some lift $F$ of $f$. Then, there exists $\mathcal{M}$ minimal set of $f$ such that $\rho_{\mathcal{M}}(F)$ is a non-trivial segment.

\end{cor}

Let $\mathcal{A}\subset \homeo$ be defined as $\mathcal{A}=\{f\in\homeo:\textrm{int}(\rho(F))\neq \phi\mbox{ for some $F$ lift of $f$}\}$. As shown in \cite{MiZi2} this set is open in $\homeo$.

\begin{cor}\label{abun1}

There exists an open and dense set $\mathcal{D}'\subset \mathcal{A}$, such that any $f\in\mathcal{D}'$ has a minimal set $\mathcal{M}$ with $\rho_{\mathcal{M}}(F)$ given by a non trivial segment.

\end{cor}

\begin{proof}

Let $\mathcal{D}$ be the set constructed in Section \ref{TEO1}, and $\mathcal{D}':=\mathcal{D}\cap\mathcal{A}$. Thus by construction of $\mathcal{D}$ for any element $g\in\mathcal{D}'$ there exist a fitted axiom A diffeomorphism $f\in\mathcal{D}'$ and  a homotopic to the identity map $h:\T^2\rightarrow \T^2$ such that $h$ semi-conjugates $g$ with $f$ ($h\circ g=f\circ h$).

\ Due to the corollary \ref{abun0} we have the existence of a minimal set $\mathcal{M}'$ of $f$ such that $\rho_{\mathcal{M}'}(F)$ is a non trivial segment for some lift $F$ of $f$. Le us define the $g$-invariant set $\mathcal{M}''=h^{-1}(\mathcal{M}')$, and consider a minimal set $\mathcal{M}\subset \mathcal{M}''$ of $g$. Hence, since $\mathcal{M}'$ is a minimal set of $f$ we have necessarily that $h(\mathcal{M})=\mathcal{M}'$. Thus by the properties of the rotation set we have that $\rho_{\mathcal{M}}(G)$ is a non-trivial segment.

\end{proof}

Finally, inspired by the result given in \cite{MiZi2} we pose the following question:

\ Given $f\in\mathcal{A}$ and a continuum $C\subset int(\rho(F))$, does there exists a minimal set $\mathcal{M}$ of $f$ such that $\rho_{\mathcal{M}}(F)=C$?

\end{section}

\end{document}